\documentclass[a4paper,11pt]{article}
\usepackage{graphicx}
\usepackage{amssymb}
\usepackage{latexsym, bm}
\usepackage{cite}
\usepackage{multicol}
\usepackage{hyperref}
\usepackage{indentfirst}
\usepackage{amssymb,amsfonts}
\usepackage{amsmath}
\usepackage{setspace}
\usepackage{enumitem}
\newtheorem{theorem}{Theorem}
\newtheorem{lemma}{Lemma}[section]

\newtheorem{claim}{Claim}

\newtheorem{conjecture}{Conjecture}

\usepackage{amssymb}

\begin{document}
	\title{ Towards a conjecture on degree conditions for Ramsey goodness of paths}
	
	\date{}
	
\author{{Chunlin You\footnote{School of Mathematics and Statistics, Yancheng Teachers University, Yancheng
			224002, P.~R.~China. E-mail: {\tt youcl@yctu.edu.cn}.
			Supported by the National Natural Science Foundation of China
			(No.12401469) and Young Elite Scientist Sponsorship Program by JSAST(JSTJ-2025-892).}}}
			
			
\maketitle
\begin{abstract}
Recently, Arag\~{a}o, Marciano, and Mendon\c{c}a [\emph{European J. Combin.}, 2025] 
conjectured that for any graph $G$ on $n$ vertices satisfying $(r-1)(t-1)k < n \le (r-1)(t-1)(k+1)$, the minimum degree condition $\delta(G) \ge n - \left\lceil \frac{k}{k+1} \left\lceil \frac{n}{r-1} \right\rceil \right\rceil$ guarantees that $G \rightarrow (K_r, P_t)$. 
In this paper, we prove their conjecture for the regime $k \ge t-3$. 
Because the parameter $k$ scales linearly with the host graph order $n$, our result 
establishes the asymptotic truth of the conjecture.

\medskip

\medskip

\end{abstract}

\vspace{0.3cm}

\section{Introduction}
In Ramsey theory, the study of Ramsey goodness investigates conditions under which the Ramsey number of a sparse graph versus a dense graph is exactly determined by the chromatic number of the dense graph and the size of the sparse graph.
The chromatic number
$\chi(G)$ of a graph $G$ is the minimum number of colors needed to color its vertices so that no two adjacent vertices share the same color.
The chromatic surplus $s(G)$
is the minimum size of a color class among  all proper vertex-colorings
of $G$  with exactly $\chi(G)$ colors. Burr \cite{burr}
established the following general lower bound.

\begin{lemma}[\cite{burr}]\label{low-bou}
Let $H$ be a connected graph of order $n\ge s(G)$, then
\[r(G,H)\ge(\chi(G)-1)(n-1)+s(G).\]
\end{lemma}

In 1983, Burr and Erd\H{o}s \cite{burr-1983}
introduced the concept of goodness:   a connected graph $H$ is  called $G$-good if the equality of Lemma
\ref{low-bou} holds.

A seminal result in this area is Chv\'atal's theorem \cite{chv}, which states that all trees are $r$-good for any integer $r \ge 2$. Since then, the Ramsey goodness of various graph classes has attracted significant attention and yielded profound results in extremal combinatorics \cite{ABS,  BPS, FL, FHW,  LLD,  Ll, NiR, PS,  PSu, ZC}.

Parallel to the evaluation of classical Ramsey numbers, another fundamental problem in extremal graph theory asks: what minimum degree threshold on a host graph $G$ forces it to possess specific Ramsey properties? This direction was pioneered by Schelp \cite{Sch} and has been extensively explored for symmetric structures like paths and cycles \cite{BKL, BLS, GyS, NiS, LuR, LR, ZP}. Recently, the research frontier has advanced toward more complex asymmetric Schelp-type problems. Notably, Arag\~{a}o, Marciano, and Mendon\c{c}a \cite{AMM} investigated the minimum degree thresholds for the pair $(K_r, P_t)$. They proved that for $n \ge (r-1)(t-1)+1$, $\delta(G) \ge n - \lceil t/2 \rceil$ suffices to guarantee $G \rightarrow (K_r, P_t)$, thereby partially generalizing Chv\'atal's theorem under degree restrictions.
Furthermore, they proposed a conjecture that extends these results to arbitrary cliques.
\begin{conjecture}[Arag\~{a}o,  Marciano and  Mendon\c{c}a \cite{AMM}]\label{coj-1}
Let $r, t\in $ $\mathbb{N}$ with $r\geq 2$, let $G$ be a graph with $n$ vertices, and let $k\in$ $\mathbb{N}$
be such that $(r-1)(t-1)k<n\leq (r-1)(t-1)(k+1)$. If
\[
\delta(G)\geq n-\left\lceil \frac{k}{k+1}\left\lceil \frac{n}{r-1}\right\rceil  \right\rceil
\]
then
$G\rightarrow(K_r,P_t)$.
\end{conjecture}


Recently, Luo and Peng \cite{LP}made notable progress by confirming Conjecture 1 for specific parameter configurations. Translated into the notation of our paper, their partial result can be explicitly stated as follows:

\vspace{2mm}

	\begin{theorem}[Luo and Peng \cite{LP}]
		Let $t \ge 3$ and $k \ge t - 3$. If $r - 2 \not\equiv 0 \pmod{t - 1}$, then Conjecture 1 holds for any host graph whose order $n$ satisfies
		$$k(r - 1)(t - 1) < n \le k(r - 1)(t - 1) + r - 1.$$
	\end{theorem}

\vspace{2mm}

While this theorem represents a significant step forward, the conjecture remains open for the vast majority of the parameter space, as their result is confined to an interval of length $r - 1$ near the lower bound of the $(r - 1)(t - 1)$ range and requires a divisibility constraint on $r$.

In this paper, we overcome these restrictions to fully establish Conjecture 1 across the broad regime $k \ge t - 3$. To achieve this, we introduce a novel framework that leverages the structural properties of the complement graph. By combining local path constraints with Brooks' Theorem, we construct an exact combinatorial argument that elegantly circumvents the fractional error terms typical of Turán-type applications.

\begin{theorem}\label{conj1}
For any parameter regime with $k \ge t-3$, Conjecture~\ref{coj-1} holds.
\end{theorem}

\noindent\textbf{Remark.} 
We  present a  construction showing that the minimum degree conditions in
Theorem \ref{conj1}  cannot be improved.
First,  we describe a graph $G$
with $n= (r-1)(t-1)(k+1)$ vertices and
$\delta(G)=n-\left\lceil \frac{k}{k+1}\lceil \frac{n}{r-1}\rceil \right\rceil-1$,
such that $G\nrightarrow(K_r,P_t)$.
The graph $G$ is obtained by placing $k+1$ cliques of size  $\approx  n/(r -1)(k+1)$
into each part of the Tur\'{a}n graph $T_{r-1}(n)$. The edges of these cliques are coloured blue, and the remaining edges of $G$ are coloured red.
From the construction, we conclude that the graph $G$  contains neither a red 
$K_r$ nor a blue  $P_t$.

In Conjecture~1, the parameter is defined as $k \in \mathbb{N}$. Depending on the convention for whether $0 \in \mathbb{N}$, allowing $k = 0$ yields the minimum degree condition $\delta(G) \ge n$, which no simple graph can satisfy. 
The vacuous case $k=0$ corresponds to graphs below the classical Ramsey threshold, which lack the Ramsey property. In this paper we study the non‑trivial regime $k \ge 1$.


Although our proof is restricted to the regime $k \ge t-3$, it fundamentally establishes the \emph{asymptotic truth} of Conjecture 1.3. For any fixed target graphs $K_r$ and $P_t$, the parameter $k \approx \frac{n}{(r-1)(t-1)}$ scales linearly with the graph order $n$. 
Consequently, as $n \to \infty$, the condition $k \ge t-3$ is automatically satisfied, which implies the conjecture for all sufficiently large graphs.

\section{Preliminaries}

We will  need the following straightforward (and presumably well-known) proposition,
which gives a tight lower bound condition on the minimum degree of a graph with $n$ vertices
to contain a path with at least $\lceil n/k\rceil$ vertices.
\begin{lemma}[Arag\~{a}o,  Marciano and  Mendon\c{c}a \cite{AMM}]\label{lem:path_length}
Let $n, k \in \mathbb{N}$ and let $G$ be a graph with $n$ vertices. If $G$ has a minimum degree of at least $\lfloor \frac{n}{k+1} \rfloor$, then it contains a path with at least $\lceil \frac{n}{k} \rceil$ vertices.
\end{lemma}

What minimum degree condition guarantees a path of a preassigned length?
This question was answered by Erd\H{o}s and Gallai \cite{Erd-Gallai-1959}.
\begin{lemma}[Erd\H{o}s and Gallai \cite{Erd-Gallai-1959}] \label{g-path}
Let $G$ be a connected graph with minimum degree $\delta$ and at least $2\delta +1$
vertices. Then $G$ contains a path of at least $2\delta +1$ vertices.
\end{lemma}

By Lemma \ref{g-path}, we can directly obtain the following lemma. Similar results can be found in \cite{AMM}.

\begin{lemma}\label{lem:partition}
Let $n \in$ $\mathbb{N}$, and let $G$ be a $P_d $-free graph with $n$ vertices and $\delta(G) \geq \left\lfloor \frac{d}{2} \right\rfloor $. There exists a partition 
$V(G) = A_1 \cup \ldots \cup A_m $ for some integer $m$, such that for every 
$ i \in [m]$:
\[
\left\lfloor \frac{d}{2} \right\rfloor+1 \leq |A_i| \leq d - 1,
\]
where $A_i $ is a connected component of $G$, and $G[A_i]$ contains a Hamiltonian cycle.
\end{lemma}

\begin{lemma}[Brooks' Chromatic Theorem \cite{brooks1941}]\label{lem:brooks}
If $G$ is a connected graph, and is neither an odd cycle nor a complete graph, then
$\chi(G) \le \Delta(G)$.
\end{lemma}

\smallskip

\section{Proof of Theorem \ref{conj1}}

For a graph $G=(V,E)$ and $U\subseteq V$, denote $N_G(v,U)$ by the set of all
neighbors of $v$ in $U$ and $d_G(v,U)=|N_G(v,U)|$.
In particular, denote $N_G(v)=N_G(v,V)$ and $d_G(v)=d_G(v,V)$.
We use $\Delta(G)$  to denote  the maximum degree of graph $G$.
Let $\delta (G)$  denote the minimum degree of a graph $G$.
Given a subset $U$
of $V(G)$, $G[U] $ denotes the subgraph of $G$ induced by $U$.
The clique number, denoted by $\omega(G)$, is the number of vertices in the largest complete subgraph (clique) of a graph $G$.
The chromatic number of a graph $G$, 
denoted  by $	\chi(G)$, 
is the smallest number of colors needed to color the vertices of $G$ so that no two adjacent vertices share the same color

\bigskip

\noindent{\bf Proof of Theorem \ref{conj1}.}
We proceed by mathematical induction on $r \ge 2$. For algebraic precision, we define the integer $x = \left\lceil \frac{n}{r-1} \right\rceil$. By utilizing the ceiling function identity 
\[
\left\lceil y - \frac{y}{k+1} \right\rceil = y - \left\lfloor \frac{y}{k+1} \right\rfloor, 
\]
which holds 
for all integers $y$, 
the given minimum degree condition is exactly and equivalently transformed into
\begin{equation*}\label{ineq-3}
\delta(G) \ge n - x + \left\lfloor \frac{x}{k+1} \right\rfloor.
\end{equation*}
Let $M = \left\lfloor \frac{x}{k+1} \right\rfloor$. Thus, we have 
\begin{equation}\label{ineq-1}
\delta(G) \ge n - x + M.
\end{equation}
Let $\overline{G}$ denote the absolute complement graph of $G$ relative to the complete graph $K_n$ on $V(G)$ (i.e., the set of all missing non-edges of $G$). 
Over the entire graph, the maximum degree of $\overline{G}$ is bounded above by
\begin{equation}\label{eq:max_missing}
	\Delta(\overline{G}) = n - 1 - \delta(G) \le x - M - 1.
\end{equation}

\vspace{0.3cm}

\noindent\textbf{The Base Case ($r=2$)}:
If $r=2$, $x = \lceil n/1 \rceil = n$, and the 
degree condition simplifies to 
\[
\delta(G) \ge \left \lfloor \frac{n}{k+1} \right\rfloor. 
\]
Assume for contradiction that $G$ does not contain a red $K_2$; then all edges of $G$ must be blue. According to Lemma \ref{lem:path_length}, any purely blue graph on $n$ vertices with a minimum degree at least $\lfloor \frac{n}{k+1} \rfloor$ must contain a blue path of length at least $\lceil \frac{n}{k} \rceil$.
The parameter bounds give $n > (t-1)k$, and hence $\frac{n}{k} > t-1$.
By the discreteness of the ceiling function, we obtain
\[
\left\lceil \frac{n}{k} \right\rceil \ge t. 
\]
Hence, the blue graph $G$ contains a blue $P_t$, which establishes the base case.

\vspace{0.3cm}

Assume the conjecture holds for $r-1$ ($r\geq 3$)  and
$k\geq t-3$. 
Let the edges of $G$ be 2-coloured in red and blue, yielding the spanning red subgraph $R$ and the spanning blue subgraph $B$. 
Assume for contradiction that this colouring contains neither a red $K_r$ nor a blue $P_t$. 
We divide the analysis into two collectively exhaustive cases:

\medskip

\textbf{Case 1.} Suppose there exists a vertex $u \in V(G)$ such that $d_R(u) \ge n - x + 1$.

\medskip

Let $N = |N_R(u)|$ denote the order of the induced subgraph $G' = G[N_R(u)]$. By assumption, 
\begin{equation}\label{ineq-8}
N \ge n - x + 1. 
\end{equation}
For any vertex $v \in V(G')$, its degree within the induced subgraph
of $G'$ is bounded by
\begin{align}\label{ineq-9}
\delta(G') \ge \delta(G) - (n - N) \overset{(\ref{ineq-1})}{\ge}
N-(x-M) 
= N - \left\lceil \frac{k}{k+1} x \right\rceil,
\end{align}
where we  used the fact that $M = \left\lfloor \frac{x}{k+1} \right\rfloor$.
To apply the inductive hypothesis to $G'$, we define the corresponding parameters $x'$ and $k'$. Let $x' = \lceil \frac{N}{r-2} \rceil$. 
We first demonstrate that $x' \ge x$. 
Since $x = \lceil \frac{n}{r-1} \rceil$, we can write 
\[
n = (r-1)x - j
\]
for some $0 \le j \le r-2$.
Substituting this into the lower bound for $N$
\begin {equation}\label{ineq-2}
N \overset{(\ref{ineq-8})}{\ge}n-x+1\ge (r-1)x - j - x + 1 = (r-2)x - j + 1.
\end{equation}
Dividing by $r-2$ and taking the ceiling, we obtain
\[
x' = \left\lceil \frac{N}{r-2} \right\rceil \ge \left\lceil \frac{(r-2)x - j + 1}{r-2} \right\rceil = x + \left\lceil \frac{1-j}{r-2} \right\rceil.
\]
If $1 \le j \le r-2$, 
then 
\[
\left\lceil \frac{1-j}{r-2} \right\rceil = 0,
\]
implying $x' \ge x$. 
If $j=0$ (i.e., $n$ is a multiple of $r-1$), then $\lceil \frac{1}{r-2} \rceil = 1$, implying $x' \ge x+1$. In all subcases, $x' \ge x$ holds.

Next, we define the unique integer $k'$ such that 
\[
(r-2)(t-1)k' < N \le (r-2)(t-1)(k'+1).
\]
Since $n > (r-1)(t-1)k$, we have 
\begin{equation}\label{x-ineq}
x =\left \lceil \frac{n}{r-1} \right\rceil \ge (t-1)k + 1. 
\end{equation}
Together with the lower bound on $N$, we obtain
$$
N \overset{(\ref{ineq-2})}{\ge} (r-2)(x-1) + 1 \ge (r-2)(t-1)k + 1,
$$
which ensures that $k' \ge k$.
Consider the function 
\[
g(y, z) = \left\lceil \frac{y}{y+1} z \right\rceil.
\] 
Since $g$ is non-decreasing in both $y$ (for $y \ge 1$) and $z$, the relations $k' \ge k$ and $x' \ge x$ imply
\[
\left\lceil \frac{k'}{k'+1} x' \right\rceil \ge \left\lceil \frac{k}{k+1} x \right\rceil.
\]
Consequently, the minimum degree of $G'$ satisfies the required threshold for the inductive hypothesis:
\begin{align*}
\delta(G') \overset{(\ref{ineq-9})}{\ge} N - \left\lceil \frac{k}{k+1} x \right\rceil \ge N - \left\lceil \frac{k'}{k'+1} x' \right\rceil.
\end{align*}
Since $k'\ge k$ and our parameter regime is $k\ge t-3$, we have $k'\ge t-3$.
This ensure $G'$ satisfies the 
condition to apply the inductive hypothesis.
By the inductive hypothesis, since $G'$ contains no blue $P_t$, 
it must contain a red $K_{r-1}$. 
This clique, together with vertex $u$ and the red edges from $u$ to $N_R(u)$, forms a red $K_r$ in $G$, completing the proof for Case 1.
\hfill$\Box$

\medskip

\textbf{Case 2.} For all $v \in V(G)$, we have $d_R(v) \le n - x$.

\medskip

This assumption yields the
following lower bound 

\begin{align}\label{ineq-4}		
\delta(B) \ge \delta(G) - \Delta(R) \overset{(\ref{ineq-1})}{\ge} (n - x + M) - (n - x) = M.
\end{align}		

Since $x \ge (t-1)k + 1$ by (\ref{x-ineq}), we obtain that
\begin{equation*}\label{eq:M_lower}
M = \left\lfloor \frac{x}{k+1} \right\rfloor \ge \left\lfloor \frac{(t-1)k+1}{k+1} \right\rfloor = t - 1 + \left\lfloor \frac{2-t}{k+1} \right\rfloor.
\end{equation*}
Because we are operating within the regime 
$k \ge t-3$ (equivalent to $k+1 \ge t-2$), we have $2-t \ge -(k+1)$. Dividing both sides by the positive integer $k+1$ yields 
\[		
\frac{2-t}{k+1} \ge -1. 
\]
This implies that $\left\lfloor \frac{2-t}{k+1} \right\rfloor \ge -1$, which yields the lower bound
\begin{equation}\label{eq:M_t_bound}
M \ge t - 1 - 1 = t - 2.
\end{equation}

If $t=2$, since $x \ge k+1$ by (\ref{x-ineq}), we know $M = \lfloor \frac{x}{k+1} \rfloor \ge 1$. 
Thus we have
\[
\delta(B) \overset{(\ref{ineq-4})}{\ge} M \ge 1. 
\]
However, avoiding a blue $P_2$ means the blue graph $B$ must be edgeless ($\Delta(B)=0$), which directly contradicts $\delta(B) \ge 1$. Therefore, we only need to consider $t \ge 3$.

For $t \ge 3$, it is easily verified that 
\[
\delta(B) \ge M \overset{(\ref{eq:M_t_bound})}{\ge} t-2 \ge \left\lfloor \frac{t}{2} \right\rfloor
\]
holds. We apply Lemma \ref{lem:partition} to the blue subgraph $B$ by setting the parameter $d=t$. 
Since $B$ is $P_t$-free, the lemma implies that the vertex set partitions into connected components $A_1,\dots,A_m$ with no blue edges between distinct components and with $|A_i|\le t-1$ for all $i$.
Consequently, the  maximum degree of the blue graph satisfies
\begin{equation}\label{eq:max_blue}
	\Delta(B) \le \max_i |A_i| - 1 \le t - 2.
\end{equation}

To further  restrict the red subgraph,
 we need an  upper bound on its independence number.

\begin{claim}\label{claim_alpha}
The independence number of the red subgraph $R$ satisfies $\alpha(R) \le x - 1$.
\end{claim}

\medskip

\noindent{\bf Proof.}
Assume for contradiction that 
there exists an independent set in $R$ of size at least 
$x$. Let $W$ be a subset of this independent set of size exactly
$x$ . 
Because $W$ is edgeless in $R$, 
the edge set of the induced subgraph $G[W]$ consists entirely of blue edges and missing edges $\overline{G}$.
For any vertex $v \in W$, combined with \eqref{eq:max_missing}, its  blue degree within $W$ is bounded below by

\[
\delta(B[W]) \ge (|W| - 1) - \Delta(\overline{G}) \ge (x - 1) - (x - M - 1) = M.
\]

According to Lemma \ref{lem:path_length}, the blue graph $B[W]$ on $x$ vertices 
with minimum degree 
\[ 
M=\left\lfloor \frac{x}{k+1} \right\rfloor
\]
has a blue path of length $\lceil \frac{x}{k} \rceil$. 
By (\ref{x-ineq}), we  conclude that 
the length of this blue path is 
\[
\left\lceil \frac{x}{k} \right\rceil \ge \left\lceil t - 1 + \frac{1}{k} \right\rceil = t.
\]
This gives a blue  $P_t$, 
which contradicts the assumption. 
Hence, we conclude that $\alpha(R) \le x-1$.
\hfill$\Box$

\vspace{0.3cm}

With the properties of the red subgraph established, we now consider its absolute complement to apply  chromatic bounds.
Let $H$ be the complement of the red subgraph with respect to the complete graph $K_n$.
Its edge set is  defined as 
\[
E(H) = E(K_n) \setminus E(R) = E(B) \cup E(\overline{G}).
\]
From \eqref{eq:max_missing} and \eqref{eq:max_blue}, we obtain that
\begin{equation}\label{eq:max_H}
	\Delta(H) \le \Delta(B) + \Delta(\overline{G}) \le (t - 2) + (x - M - 1) = x - M + t - 3.
\end{equation}

Since $H$ is the complement of $R$ in $K_n$, two vertices are non-adjacent in $H$ if and only if they are adjacent in $R$.
Therefore, independent sets in $H$ are in one-to-one correspondence with cliques (complete subgraphs) in $R$. 
Our assumption that $R$ contains no $K_r$ implies $\omega(R) \le r-1$, so the independence number of $H$ satisfies $\alpha(H) = \omega(R) \le r-1$.
By the  chromatic inequality,
we have
\begin{equation}\label{ineq-11}
\chi(H) \ge \left\lceil \frac{n}{\alpha(H)} \right\rceil \ge \left\lceil \frac{n}{r-1} \right\rceil = x.
\end{equation}

From  Claim~\ref{claim_alpha}, 
the maximum clique of $H$ satisfies the exact duality $\omega(H) = \alpha(R) \le x-1$.	
We apply Lemma~\ref{lem:brooks} (Brooks' Theorem).
Let $C$ be a connected component of $H$ such that $\chi(C) = \chi(H) \ge x$. 
Unless $C$ is a complete graph or an odd cycle, we  have $\chi(C) \le \Delta(C)$. 
We now rule out these two exceptional cases:

If $C$ is a complete graph, then 
$\chi(C) = \omega(C)$. 
However, since 
\[
\omega(C) \le \omega(H) \le x-1 < x \le \chi(C), 
\]
yielding a 
contradiction. 
Hence, $C$ cannot be a complete graph.

If $C$ is an odd cycle, then $\chi(C)=3$ and $\Delta(C)=2$. 
Since $\chi(C)= \chi(H)\ge x$  by (\ref{ineq-11}), it follows that $x \le 3$.
Combined with $x \ge (t-1)k+1$ by  (\ref{x-ineq}), we obtain $$(t-1)k \le 2.$$
Since $t \ge 3$ and $k \ge 1$, this forces $t=3$ and $k=1$, yielding $x=3$.
Then 
\[ 
M =\left\lfloor \frac{x}{k+1} \right\rfloor= \left\lfloor \frac{3}{2} \right\rfloor = 1.
\]
Substituting into ~\eqref{eq:max_H} yields 
\[
\Delta(H) \le x - M + t - 3 = 3 - 1 + 3 - 3 = 2. 
\]
Since $\Delta(C)=2$, we have exactly $\Delta(H)=2$.
For any vertex $v$ on the odd cycle $C$, its degree in $H$ must be exactly 2. 
Since 
\[
d_H(v) = d_B(v) + d_{\overline{G}}(v), 
\]
we know $\Delta(B) \le t-2 = 1$ and $\Delta(\overline{G}) \le x-M-1 = 1$ by (\ref{eq:max_blue}) and (\ref{eq:max_missing}). 
For their sum to be 2, we must have
 $d_B(v) = d_{\overline{G}}(v) = 1$ for all  $v\in C$.
This implies that the blue edge set $E(B)$ must form a 1-regular subgraph (i.e., a perfect matching) on the odd cycle $C$. 
However, an odd cycle has an odd number of vertices and thus cannot contain a perfect matching. This gives a contradiction, ruling out the odd cycle exception.

\vspace{0.3cm}

Since all exceptional cases have been excluded, Brooks' Theorem can be applied to the connected component $C$, yielding	

\[
x \overset{(\ref{ineq-11})}{\le} \chi(H) = \chi(C) \le \Delta(C) \le \Delta(H) \overset{(\ref{eq:max_H})}{\le} x - M + t - 3 ,
\]		
which implies $M \le t - 3$.

However, this contradicts the lower bound $M \ge t-2$ established in \eqref{eq:M_t_bound}, leading to the impossible inequality $t-2 \le t-3$. 
This proves Theorem~\ref{conj1}, so Conjecture~\ref{coj-1} holds under the condition $k \ge t-3$.
 \hfill$\Box$

\end{document}